\documentclass{amsart}

\usepackage{amsmath, amsthm, amssymb, amstext, amsfonts}
\usepackage{enumerate}
\usepackage{graphicx}
\usepackage[applemac]{inputenc}

\theoremstyle{plain}

\newtheorem{thm}{Theorem}[section]
\newtheorem{lem}[thm]{Lemma}
\newtheorem{cor}[thm]{Corollary}
\newtheorem{prop}[thm]{Proposition}

\newtheorem{prob}{Problem}
\newtheorem*{conjTA}{Tree Alternative Conjecture}

\theoremstyle{definition}

\newtheorem{exam}[thm]{Example}

\newcommand{\sm}{\ensuremath{\smallsetminus}}

\newcommand{\Emb}{\rm{Emb}}

\newcommand{\es}{\ensuremath{\emptyset}}

\newcommand{\sub}{\subseteq}

\newcommand{\comment}[1]{}

\newcommand{\se}{self-em\-bed\-ding}

\newcommand{\nat}{{\mathbb N}}

\newcommand{\ganz}{{\mathbb Z}}

\newcommand{\DF}{\ensuremath{\mathcal D}}

\newcommand{\LF}{\ensuremath{\mathcal L}}

\newcommand{\TF}{\ensuremath{\mathcal T}}

\renewcommand{\L}{\ensuremath{\mathcal L}}
\newcommand{\D}{\ensuremath{\mathcal D}}

\def\?#1{\vadjust{\vbox to 0pt{\vss\vskip-8pt\leftline{%
     \llap{\hbox{\vbox{\pretolerance=-1
     \doublehyphendemerits=0\finalhyphendemerits=0
     \hsize16truemm\tolerance=10000\small
     \lineskip=0pt\lineskiplimit=0pt
     \rightskip=0pt plus16truemm\baselineskip8pt\noindent
     \hskip0pt        
     #1\endgraf}\hskip7truemm}}}\vss}}}

\newenvironment{txteq*}
  {
    \begin{equation*}
    \begin{minipage}[c]{0.85\textwidth} 
    \em                                
  }
  {\end{minipage}\end{equation*}\ignorespacesafterend}

\begin{document}

\title{Self-embeddings of trees}
\author{Matthias Hamann}
\address{Matthias Hamann, Alfr\'ed R\'enyi Institute of Mathematics, Hungarian Academy of Sciences, Budapest, Hungary}
\thanks{The author was partially supported by the European Research Council under the European Union's Seventh Framework Programme (FP7/2007-2013) / ERC grant agreement n$^\circ$ 617747.}

\begin{abstract}
We prove a fix point theorem for monoids of \se s of trees.
As a corollary, we obtain a result by Laflamme, Pouzet and Sauer that a tree either contains a subdivided binary tree as a subtree or has a vertex, and edge, an end or two ends fixed by all its \se s.
\end{abstract}

\maketitle

\section{Introduction}

For a group acting on a graph there is always the following choice: either it fixes a point or the group contains a free subgroup $\ganz\ast\ganz$.
More precisely the following statements are true for a group $\Gamma$ acting on a graph~$G$.
\begin{enumerate}[$\bullet$]
\item Every automorphism of~$G$ is either elliptic, hyperbolic, or parabolic;
\item $\Gamma$ fixes either a bounded subset of~$G$ or a unique limit point of~$\Gamma$ in its end space, or $G$ has precisely two limit points of~$\Gamma$, or $\Gamma$ contains two hyperbolic elements that freely generate a free subgroup.
\item there are either none, one, two or infinitely many limit points of~$\Gamma$;
\item there are either none, two or infinitely many hyperbolic limit points of~$\Gamma$;
\item the hyperbolic limit set of~$\Gamma$ is dense in the limit set of~$\Gamma$;
\item if the limit set of~$\Gamma$ is infinite, then it is a perfect set.
\end{enumerate}
We refer to \cite{H-AutoAndEndo, GroupsOnMetricSpaces, J-FiniteFixedSets, EoG, P-BilatDenseness, Woess-Amenable, W-FixedSets} for all these theorems.

All these results carry over almost verbatim to monoids of \se s of trees and in this paper, we will prove these analogues.
As a corollary of our results, we obtain a result by Laflamme et al.~\cite{LPS-Twins} that a tree either contains a subdivided binary tree or has a vertex, an edge, or a set of at most two ends fixed by all \se s.

Laflamme et al.\ used their result to varify the tree alternative conjecture of Bonato and Tardif~\cite{BonTar06} in various situations.
We will discuss the connection of our results with that conjecture in Section~\ref{sec_TAC}.
We will also give an outlook on the possible generalisation of the present results to general graphs in Section~\ref{sec_Graphs}.

\section{Main results}

Let $T$ be a tree.
A \emph{\se} is an injective map of $V(T)$ into itself that preserves the adjacency relation.

A \emph{ray} is a one-way infinite path in~$T$ and two rays are \emph{equivalent} if they have for every $v\in V(T)$ subrays in the same component of $T-v$.
This is an equivalence relation whose classes are the \emph{ends} of~$T$.
By $\Omega(T)$ we denote the set of ends of~$T$.

Note that each \se\ of~$T$ maps rays to rays and equivalent rays to equivalent rays. Thus, it induces a map from $\Omega(T)$ into itself.

By $\Emb(T)$ we denote the monoid of all \se s of~$T$.
We call a \se\ $g\in \Emb(T)$
\begin{enumerate}[$\bullet$]
\item \emph{elliptic} if it fixes a non-empty finite subtree of~$T$;
\item\emph{hyperbolic} if it is not elliptic and if it fixes precisely two ends;
\item \emph{parabolic} if it is not elliptic and if it fixes precisely one end.
\end{enumerate}

Note that Halin~\cite[Lemma~2]{H-AutoAndEndo} proved that any automorphism of a finite tree fixes either a vertex or an edge.
It follows that a \se\ is elliptic if and only if it fixes either a vertex or an edge.
Halin proved the following theorem.

\begin{thm}\label{thm_Halin}\cite[Theorem~5]{H-AutoAndEndo}
Let $g$ be a \se\ of a tree~$T$.
Then either $g$ fixes a vertex or an edge or there is a ray $R$ with $g(R)\subsetneq R$.
This ray can be extended to either a $g$-invariant double ray or a maximal ray $R'$ with $g(R')\subsetneq R'$.\qed
\end{thm}

Note that any $g\in\Emb(T)$ that is not elliptic fixes at least one end due to Halin's theorem, namely that end that contains the ray~$R$ of Theorem~\ref{thm_Halin}.
If $R$ can be extended to a $g$-invariant double ray~$R'$, then both ends defined by~$R'$ are $g$-invariant.
Furthermore, no other end $\omega$ can be fixed, since the unique ray with precisely one vertex on~$R'$ that lies in~$\omega$ will be mapped to a disjoint ray.
Since disjoint rays lie in distinct ends of a tree, $g$ is hyperbolic.
A similar argument shows that in the case that $R'$ is a ray, $g$ is parabolic.
Thus, we obtain as a corollary of Theorem~\ref{thm_Halin} the following.

\begin{cor}
Every \se\ of a tree is either elliptic, hyperbolic, or parabolic.\qed
\end{cor}

Note that for any non-elliptic \se~$g$ all rays that are preserved by~$g$ are equivalent, since otherwise the double ray between the two ends these non-equivalent rays lie in had non-equivalent tails $R_1$ and~$R_2$ with $g(R_1)\subsetneq R_1$ and $g(R_2)\subsetneq R_2$, which is impossible. 
We call this end the \emph{direction $g^+$} of~$g$.
If $g$ is hyperbolic, we denote by~$g^-$ the unique $g$-invariant end other than~$g^+$.

Since we also talk about convergence to ends, we need a topology on trees with their ends.
For this we consider the tree as a $1$-complex.
The sets $C$ where $C$ is a component of $T-x$ for some vertex~$x$ that contain an end $\omega$ form a neighbourhood basis of~$\omega$.\footnote{For a general approach to locally finite graphs with their ends seen as topological spaces, we refer to Diestel~\cite{TopSurveyI, TopSurveyII}.}
For $x\in V(T)$ and $\omega\in\Omega(T)$, we denote by $C(T-x,\omega)$ that component of $T-x$ that contains~$\omega$, that is, that component of~$T-x$ that contains a sequence of vertices converging to~$\omega$.

Trees with their ends are \emph{projective}: whenever $(x_i)_{i\in\nat}$ and $(y_i)_{i\in\nat}$ are sequences such that the distances $d(x_i,y_i)$ are bounded, then $(x_i)_{i\in\nat}$ converges to an end~$\omega$ if and only if $(y_i)_{i\in\nat}$ converges to~$\omega$, see~\cite{GroupsOnMetricSpaces,W-FixedSets}.
It follows that for any $v\in V(T)$ and any non-elliptic \se, the sequence $(g^i(v))_{i\in\nat}$ converges to~$g^+$.

The following lemma gives a condition under which a \se\ is non-elliptic.

\begin{lem}\label{lem_nonElliptic}
Let $g$ be a \se\ of a tree~$T$.
If there is an edge $xy$ such that $y$ and $g(x)$ separate $x$ and $g(y)$, then $g$ is non-elliptic and $xy$ lies on the maximal (double) ray of~$T$ that is preserved by~$g$.
\end{lem}

\begin{proof}
If $g$ is elliptic, then it fixes a subtree $T'$ on either one or two vertices.
Since $d(x,T')=d(g(x),T')$ and $d(y,T')=d(g(y),T')$, we conclude that $T'$ lies in that component of $T-\{x,g(x)\}$ that contains~$y$ and also in that component of $T-\{y,g(y)\}$ that contains~$g(x)$.
But then we have $d(y,T')<d(x,T')$ and $d(g(x),T')<d(g(y),T')$.
This is a contradiction as \se s preserve the distance function.

So $g$ is not elliptic and hence preserves a maximal (double) ray~$R$ by Theorem~\ref{thm_Halin}.
If $xy$ lies not on~$R$, then if $d(x,R)<d(y,R)$, then also $d(g(x),R)<d(g(y),R)$, which is a contradiction to the assumption that $y$ and~$g(x)$ separate $x$ and $g(y)$.
Similarly, we obtain a contradiction if $d(x,R)>d(y,R)$.
Thus, $xy$ lies on~$R$.
\end{proof}

Let $M$ be a submonoid of~$\Emb(T)$.
The \emph{limit set} $\L(M)$ of~$M$ is the set of accumulation points of~$\{g(v)\mid g\in M\}$ in~$\Omega(T)$ for any $v\in V(T)$.
Note that $\L(M)$ is independent from the choice of~$v$ by projectivity.
By $\D(M)$ we denote the set of all directions of non-elliptic elements in~$M$.

\begin{thm}\label{thm_DenseNumber}
Let $T$ be a tree and $M$ a submonoid of $\Emb(T)$.
\begin{enumerate}[{\em (i)}]
\item\label{itm_DenseNumber1} If $|\L(M)|\geq 2$, then $\D(M)$ is dense in~$\L(M)$.
\item\label{itm_DenseNumber2} The set $\L(M)$ has either none, one, two, or infinitely many elements.
\item\label{itm_DenseNumber3} The set $\D(M)$ has either none, one, two, or infinitely many elements.
\end{enumerate}
\end{thm}

\begin{proof}
To prove~(\ref{itm_DenseNumber1}), let $\omega\in\LF(M)$ and $v\in V(T)$.
Then there is a sequence $(g_i)_{i\in\nat}$ in~$M$ such that $(g_i(v))_{i\in\nat}$ converges to~$\omega$.
If there are infinitely many non-elliptic among the $g_i$, then we may assume that all $g_i$ are non-elliptic.
For $i\in\nat$, let $R_i$ be the unique maximal (double) ray of~$T$ that is preserved by~$g_i$, which exists by Theorem~\ref{thm_Halin}.
Then we have $d(v,R_i)\geq d(g_i(v),R_i)$.
Let $u\in V(T)$.
If infinitely many $R_i$ have a tail in $g_i^+$ that lies in $C(T-u,\omega)$, we conclude that their directions $g_i^+$ converge to~$\omega$.
Let us suppose that only finitely many $R_i$ have a tail in $C(T-u,\omega)$.
Since $(g_i(v))_{i\in\nat}$ converges to~$\omega$, there are infinitely many $i\in\nat$ with $d(g_i(v),u)>d(u,v)$.
For all of those we conclude
$$\begin{array}{lll}
d(v,R_i)&\leq&d(v,u)+d(u,R_i)\\
&<&d(g_i(v),u)+d(u,R_i)\\
&=&d(g_i(v),R_i).
\end{array}$$
This contradiction to $d(v,R_i)\geq d(g_i(v),R_i)$ shows that $\omega$ lies in the closure of $\DF(M)$ if infinitely many $g_i$ are non-elliptic.

By considering an finite subsequence of $(g_i)_{i\in\nat}$, we may assume that all $g_i$ are elliptic.
Similarly, let $\omega'\in\LF(M)$ with $\omega'\neq\omega$ and let $(h_i)_{i\in\nat}$ be a sequence in~$M$ such that $(h_i(v))_{i\in\nat}$ converges to~$\omega'$.
Let $x\in V(T)$ lie on the unique double ray between $\omega$ and~$\omega'$ and let $x'$ be the neighbour of~$x$ that separates~$x$ from~$\omega'$.

If there are infinitely many $h_i$ elliptic, we may assume, by looking at a subsequence, that all of them are elliptic.
Also by taking subsequences, we may assume that all $g_i(x)$ lie in $C(T-x,\omega)$ and all $h_i(x)$ and $h_ig_i(x)$ lie in $C(T-x,\omega')$.
We consider the sequence $(g_ih_i)_{i\in\nat}$.
Then $x$, $g_i(x)$, and $g_ih_i(x)$ separate $x'$ from $g_ih_i(x')$.
By Lemma~\ref{lem_nonElliptic}, we obtain that $g_ih_i$ is not elliptic.
We claim that $(g_ih_i(v))_{i\in\nat}$ converges to~$\omega$.
To see this, let $y\in V(T)$.
Then there is some $n\in\nat$ such that for all $i\geq n$ we have that $g_i(x)$ lies in $C(T-y,\omega)$.
But as $g_i(x)$ separates $x$ and $g_ih_i(x)$, also $g_ih_i(x)$ lies in~$C(T-y,\omega)$.
Thus, $(g_ih_i(x))_{i\in\nat}$ converges to~$\omega$.

So at most finitely many $h_i$ are elliptic.
Again, we consider a subsequences such that all $h_i$ are not elliptic and such that all $g_i(x)$ lie in $C(T-x,\omega)$ and all $h_i(x)$ and $h_ig_i(x)$ lie in $C(T-x,\omega')$.
Let $R_i$ be the maximal (double) ray $R$ that is preserved by~$h_i$
First assume that $g_i(x')$ lies on~$R_i$.
As $x,x'$ separate $g_i(x')$ from $h_ig_i(x')$, we conclude that $x$ and~$x'$ lies on~$R$ as well and we have that $h_ig_i(x')$ and $x'$ separate $h_ig_i(x)$ from $x$.
We consider $g_ih_ig_i$ is this situation.
Since $g_ih_ig_i(x')$ and $g_i(x')$ separate $g_ih_ig_i(x)$ from $g_i(x)$, we conclude that $f_i:=g_ih_ig_i$ is not elliptic by Lemma~\ref{lem_nonElliptic}.

If $g_i(x')$ does not lie on~$R_i$, but $x$ lies on~$R_i$, set $f_i:=g_ih_i$.
Then $x'$ and $h_ig_i(x)$ separate $h_ig_i(x')$ from $x$.
So $g_i(x')$ and $g_ih_ig_i(x)$ separate $g_i(x)$ from $g_ih_ig_i(x')$.
We conclude by Lemma~\ref{lem_nonElliptic} that $f_i$ is not elliptic.

If neither $g_i(x')$ nor~$x$ lie on~$R_i$, set $f_i:=g_ih_i$.
Lemma~\ref{lem_nonElliptic} implies that there is some double ray in~$T$ that contains $h_i(x),h_i(x'),x',x,g_i(x),g_i(x'),g_ih_i(x'),g_ih_i(x)$ in this particular order.
This implies that $f_i$ is not elliptic.

So we obtain in all three cases that $f_i$ is not elliptic and that $g_i(x)$ separates $x$ from $f_i(x)$.
If we prove that $f_i(x)$ converges to $\omega$, we may replace $(g_i)_{i\in\nat}$ by $(f_i)_{i\in\nat}$ and are done by our first case, where all $g_i$ were non-elliptic.

To prove that $f_i(x)$ converges to~$\omega$, let $x_1,x_2,\ldots$ be the ray that starts at~$x$ and lies in~$\omega$.
It suffices to show that $C:=C(T-x_i,\omega)$ contains all but finitely many $f_i(x)$.
But this is a direct consequence of the facts that $C$ contains all but finitely many $g_i(x)$ and that $g_i(x)$ separates $x$ from $f_i(x)$ in all cases.
Thus, the first case proves (\ref{itm_DenseNumber1}).

To prove (\ref{itm_DenseNumber2}), let us assume $|\LF(M)|\geq 3$.
Let $f,g\in M$ be non-elliptic and such that $f^+$ is not fixed by~$g$.
These exist as every non-elliptic \se\ fixes at most two ends.
Then all $g^i(f^+)$ are distinct ends and all of them are distinct from~$g^+$ and~$g^-$.
Then $g^jf^i(v)\to g^j(f^+)$.
So all $g^j(f^+)$ lie in $\LF(M)$ and hence $\LF(M)$ contains infinitely many ends.

Finally, (\ref{itm_DenseNumber3}) is a direct consequence of (\ref{itm_DenseNumber1}) and (\ref{itm_DenseNumber2}).
\end{proof}

In the case of automorphisms\footnote{We refer to~\cite{GroupsOnMetricSpaces,W-FixedSets} for the corresponding definitions in case of automorphisms instead of \se s.}, the situation $|\LF(M)|\geq 2$ implies the existence of hyperbolic automorphisms, cp.~\cite{GroupsOnMetricSpaces,W-FixedSets}.
For \se s this is no longer the case, as we shall illustrate with the following example.
But if we add the extra assumption that some end of~$T$ is fixed by~$M$, we obtain the existence of some hyperbolic element of~$M$ in Proposition~\ref{prop_hypEx}.

\begin{exam}\label{ex_TreeWithoutHyper}
Let $T_0$ be the rooted binary tree, i.\,e.\ the tree where one vertex has degree~$2$ and all others have degree~$3$.
We add a new finite non-trivial path~$P$ to~$T_0$ that stars at the vertex of degree~$2$ and obtain a tree~$T$.
We claim that the monoid~$M$ of all \se s of~$T$ contains no hyperbolic element.
Seeking a contradiction, let us suppose that $f\in M$ is hyperbolic.
Let $R$ be the $f$-invariant double ray.
Then at all vertices~$x$ but one of~$R$, say~$u$, there is a binary tree with root~$x$ that is otherwise disjoint from~$R$.
But as some vertex of~$R$ is mapped onto~$u$, also its binary tree must be mapped onto the tree hanging of~$R$ at~$u$.
But this is impossible as the tree hanging of at~$u$ is a proper subtree of the rooted binary tree, which has a vertex of degree~$1$.
This shows that $M$ has no hyperbolic element.
\end{exam}

We note that for a hyperbolic $g\in M$ it can happen that $g^-$ does not lie in~$\LF(M)$.
To see this, we modify Example~\ref{ex_TreeWithoutHyper} a bit.

\begin{exam}\label{ex_TreeWithoutHyperContd}
Let $T_0$ be the rooted binary tree, i.\,e.\ the tree where one vertex has degree~$2$ and all others have degree~$3$.
We add a new ray~$R$ to~$T_0$ that stars at the vertex of degree~$2$ and obtain a tree~$T$.
A similar argumentation as in Example~\ref{ex_TreeWithoutHyper} shows that all hyperbolic elements $g$ of~$M$ fix the end $\omega$ that contains~$R$ and that $\omega=g^-$.
Note that there are hyperbolic elements in this situation.
\end{exam}

\begin{prop}\label{prop_hypEx}
Let $T$ be a tree and $M$ a submonoid of~$\Emb(T)$.
\begin{enumerate}[(i)]
\item If $|\LF(M)|\geq 2$ and some end is fixed by~$M$, then $M$ contains a hyperbolic \se.
\item If $|\LF(M)|=2$, then all non-elliptic elements of~$M$ are hyperbolic.
\end{enumerate}
\end{prop}

\begin{proof}
Let $|\LF(M)|\geq 2$, and let $\eta\in\LF(M)$ be fixed by~$M$.
Since $\DF(M)$ is dense in $\LF(M)$ by Theorem~\ref{thm_DenseNumber}\,(\ref{itm_DenseNumber1}), we find some $\mu\in\DF(M)$ with $\mu\neq \eta$.
Let $g\in M$ with $g^+=\mu$.
As $g$ fixes $\eta,$ it must leave the double ray between $\eta$ and~$\mu$ invariant.
Thus, $g$ is hyperbolic.

Similarly, if $|\LF(M)|=2$, all non-elliptic \se s in~$M$ must leave the double ray between the two limit points invariant.
Thus, they are hyperbolic \se s.
\end{proof}

Now we are able to prove our fixed point theorem for \se s.

\begin{thm}\label{thm_FixedPoint}
Let $T$ be a tree and $M$ a submonoid of $\Emb(T)$.
Then one of the following holds.
\begin{enumerate}[{\em(i)}]
\item\label{itm_FixedPoint1} $M$ fixes a either a vertex or an edge of~$T$;
\item\label{itm_FixedPoint2} $M$ fixes a unique element of~$\L(M)$;
\item\label{itm_FixedPoint3} $\L(M)$ consists of precisely two elements;
\item\label{itm_FixedPoint4} $M$ contains two non-elliptic elements that do not fix the direction of the other.
\end{enumerate}
\end{thm}

\begin{proof}
Let us assume that neither a vertex nor an edge is fixed by~$M$ and that no subset of $\LF(M)$ of size at most~$2$ is fixed by~$M$.
In particular, $\LF(M)$ and $\DF(M)$ are infinite by Theorem~\ref{thm_DenseNumber} and we find two non-elliptic \se s in~$M$ with distinct directions.

Let us suppose that (\ref{itm_FixedPoint4}) does not hold.
First, we show that there is a unique $\eta\in\LF(M)$ fixed by all non-elliptic \se s in~$M$.
If $M$ contains some parabolic \se\ $g$, then its direction must be fixed by all non-elliptic \se s since (\ref{itm_FixedPoint4}) does not hold.
As $g$ fixes no other end, its direction is the unique element of~$\LF(M)$ fixed by all non-elliptic \se s.

If all non-elliptic elements of~$M$ are hyperbolic, let $f,g,h\in M$ such that $g^+,h^+\notin\{f^+,f^-\}$ and $g^+\neq h^+$.
These exist as $\DF(M)$ is infinite.
As (\ref{itm_FixedPoint4}) does not hold, we know that $g$ and $h$ fix $f^+$ and hence $h^-=f^+=g^-$.
But then $g$ and $h$ satisfy (\ref{itm_FixedPoint4}) as $g$ fixes only $g^+$ and $f^+$ and $h$ fixes only $h^+$ and $f^+$.
This contradiction shows that there is a unique $\eta\in\LF(M)$ fixed by all non-elliptic \se s.

Since $\DF(M)$ is infinite, there are distinct directions $\mu,\nu\in\DF(M)\sm\{\eta\}$.
Let $f,g\in M$ be not elliptic such that the directions of $f$ and~$g$ are $\mu$ and~$\nu$, respectively.
Since $f$ fixes $\mu$ and~$\eta$, it fixes not other end, in particular it does not fix~$\nu$.
Similarly, $g$ does not fix~$\mu$.
This contradiction to the assumption that~(\ref{itm_FixedPoint4}) does not hold shows the assertion.
\end{proof}

We shall see in Theorem~\ref{thm_FreeMonoid} that we may pick the non-elliptic elements in Theorem~\ref{thm_FixedPoint}\,(\ref{itm_FixedPoint4}) so that they generate a free submonoid of~$M$ freely.

\section{Infinitely many directions}

In this section, we take a closer look at Theorem~\ref{thm_FixedPoint}~(\ref{itm_FixedPoint4}).
 But before we do that, we prove that \se s that are not elliptic converge uniformly towards~$g^+$.

\begin{lem}\label{lem_UniformConv}
Let $T$ be a tree and $g\in\Emb(T)$ be not elliptic.
For every neighbourhoods $U$ of~$g^+$ and $V$ of~$g^-$, if $g$ is hyperbolic, there exists $N\in\nat$ such that $g^n(T-V)\sub U$ for all $n\geq N$.
\end{lem}

\begin{proof}
As $U$ is a neighbourhood of~$g^+$, there is some vertex $x$ on the maximal (double) ray $R_g$ with $g(R_g)\sub R_g$ such that the component $U'$ of $T-x$ that contains $g^+$ lies in~$U$.
Similarly, if $g$ is hyperbolic, there is a vertex $y$ on~$R_g$ such that the component $V'$ of $T-x$ that contain $g^-$ lies in~$V$.
If $g$ is parabolic, let $y$ be the first vertex of the ray~$R_g$.
As $g$ is either hyperbolic or parabolic, there is some $N\in\nat$ such that $g^N(y)\in U'$.
Then $g^n(y)\in U'$ for all $n\geq N$.
Let $y'$ be the neighbour of~$y$ on~$R_g$ such that $y$ separates $y'$ and~$g^+$ and let $v\in V(T)$ be any vertex outside of~$V'$.
Then $g^n(v)$ lies in a component of~$T-g^n(y)$ that does not contain~$y'$.
Hence, we have $g^n(v)\in U'\sub U$.
\end{proof}

Now we can investigate the situation of Theorem~\ref{thm_FixedPoint}~(\ref{itm_FixedPoint4}) in more detail.

\begin{thm}\label{thm_FreeMonoid}
Let $T$ be a tree and $M$ a submonoid of $\Emb(T)$.
If $M$ contains two non-elliptic $g,h$ such that neither $g$ fixes $h^+$ nor $h$ fixes~$g^+$, then there are $m,n\in\nat$ such that $g^m$ and $h^n$ generate a free submonoid of~$M$ freely.
Furthermore, $T$ contains a subdivided $3$-regular tree.
\end{thm}

\begin{proof}
Let $R_g,R_h$ be the maximal (double) ray with $g(R_g)\sub R_g$ and $h(R_h)\sub R_h$, respectively.
Let $x\in R_g$ be with $d(x,R_h)$ minimum and, if $x\in R_h$, such that the subray of $R_h$ in~$h^+$ starting at~$x$ intersects $R_g$ only in~$x$.
Let $U_g,U_h$ be connected neighbourhoods of~$g^+$ and of~$h^+$, respectively, such that $U_g\cap (U_h\cup R_h)=\es$ and $U_h\cap (U_g\cup R_g)=\es$, such that $x\notin U_g\cup U_h$, and such that $T-U_g$ and $T-U_h$ are connected, too.
In particular, $U_g$ does not contain $h^-$ and $U_h$ does not contain $g^-$ (if they exist), and they also avoid some neighbourhood around those ends.
By Lemma~\ref{lem_UniformConv}, there are $m,n\in\nat$ such that $g^m(\{x\}\cup U_h)\in U_g$ and $h^n(\{x\}\cup U_g)\in U_h$.
As $U_g$ is connected and contains $g^+$, we have $g^m(U_g)\sub U_g$ and, analogously, $h^n(U_h)\sub U_h$.
We claim that $a:=g^m$ and $b:=h^n$ freely generate a free monoid.

Suppose they do not generate a free monoid freely.
Then there are two distinct words $w_1,w_2$ over $\{a,b\}$ that represent the same \se\ of~$T$.
We choose them such that the length of~$w_1$ is minimum.
Since $a(U_g\cup U_h)\sub U_g$ and $b(U_g\cup U_h)\sub U_h$ we conclude first that $w_i(U_g\cup U_h)\sub U_g\cup U_h)$ for $i=1,2$ and second that the first letters of $w_1$ and $w_2$ must coincide, that is, $w_1=cw_1'$ and $w_2=cw_2'$ for some $c\in \{a,b\}$ and words $w_1',w_2'$ over $\{a,b\}$.
The choice of $w_1$ being minimum implies that $w_1=w_2$ as words.
This contradiction to the assumptions shows that $a$ and~$b$ freely generate a free monoid~$M'$.

Let us now construct a subtree of~$T$ that is a subdivision of the $3$-regular tree.
Let $u\in R_g\cap U_g$ be closest to~$x$ and let $v\in R_h\cap U_h$ be closest to~$x$.
Let $P$ be the $u$-$v$ path.
Let $T_a$ be the minimal subtree of~$T$ that contains $u, a(u), a(v)$ and let $T_b$ be the minimal subtree of~$T$ that contains $v, b(u), b(v)$.
Then $T_a$ and $T_b$ are subdivisions of $K_{1,3}$.
Set
\[
T':=P\cup\bigcup_{w\in M'} w(T_a)\cup\bigcup_{w\in M'} w(T_b).
\]
It is straight forward to check that $T'$ is a subdivision of a $3$-regular tree.
\end{proof}

As a corollary of Theorems~\ref{thm_FixedPoint} and~\ref{thm_FreeMonoid}, we obtain the following, which is a theorem by Laflamme, Pouzet and Sauer~\cite{LPS-Twins}.

\begin{cor}\label{cor_LPS}\cite[Theorem 1.1]{LPS-Twins}
Let $T$ be a tree that contains no subdivision of the $3$-regular tree.
Then $\Emb(T)$ fixes either a vertex, an edge, or a set of at most two ends of~$T$.\qed
\end{cor}

Our last result of this section deals with the topology on the set of directions.

\begin{thm}
Let $T$ be a tree and $M$ a submonoid of $\Emb(T)$.
If $\LF(M)$ is infinite, then it is perfect.
\end{thm}

\begin{proof}
We have to show that $\LF(M)$ contains no isolated points.
Let us suppose that $\eta\in\LF(M)$.
As $\DF(M)$ is dense in $\LF(M)$ by Theorem~\ref{thm_DenseNumber}\,(\ref{itm_DenseNumber1}), we conclude that $\eta$ lies in $\DF(M)$.
So there is some non-elliptic $g\in M$ with $g^+=\eta$.
As $g$ fixes at most two ends but $\DF(M)$ is infinite by Theorem~\ref{thm_DenseNumber}, there is some $\mu\in DF(M)$ that is not fixed by~$g$.
Then the sequence $(g^n(\mu))_{n\in\nat}$ converges to~$\eta$.
Note that every $g^n(\mu)$ lies in $\DF(M)$: if $h_i(x)\to\mu$ for $i\to\infty$, we have $gh_i(x)\to g(\mu)$ for $i\to\infty$.
This contradicts our assumption.
\end{proof}

\section{Fixing an end}

A \se\ $g$ of~$T$ \emph{preserves an end $\omega$ forwards} if $g(R)\sub R$ for some ray $R\in\omega$ and it \emph{preserves $\omega$ backwards} if $R\sub g(R)$.
Note that any non-elliptic \se~$f$ preserves $f^+$ forwards and if $f$ is hyperbolic it additionally preserves $f^-$ backwards.
We say that $M$ \emph{preserves $\omega$ forwards} if every $g\in M$ preserves $\omega$ forwards and $M$ \emph{preserves $\omega$ backwards} if every $g\in M$ preserves $\omega$ backwards.

\begin{prop}
Let $T$ be a tree and $M$ a submonoid of $\Emb(T)$.
If $M$ preserves some end of~$T$ backwards, then all non-elliptic elements of~$M$ are hyperbolic.
\end{prop}

\begin{proof}
Let $\omega$ be an end that is preserved backwards and let $g\in M$ be not elliptic.
Since $g$ preserves $\omega$ backwards, we have $g^+\neq \omega$.
Since $g$ fixes $g^+$ and $\omega$, it is hyperbolic.
\end{proof}

Answering a question of Pouzet~\cite{P-Personal} we construct a graph that has precisely one fixed end, which is preserved backwards by all \se s.
We note that Lehner~\cite{L-Personal} also constructed an example different from ours.
Lehner's example is reproduced in \cite[Example~$6$]{LPS-Twins}.

\begin{exam}
Let $T$ be the rooted binary tree and let $x$ be its root.
To obtain $T'$, we add a new ray $R$ to~$T$ and join its first vertex with~$x$.
To all vertices that have distance $1$ modulo~$3$ to~$x$ we add $4$ new neighbours and to all vertices of distance $2$ modulo~$3$ to~$x$ we add $8$ new neighbours.
Let $T''$ be the resulting tree and let $\omega$ be the end of~$T''$ that contains~$R$.
It is straight forward to check that degree reasons imply that every \se\ fixes~$\omega$ and that no \se\ preserves $\omega$ forwards but not backwards.
\end{exam}

\section{Outlook}

In this section, we discuss open problems related to our main theorems.
Whereas the first one is a generalisation of the main theorems to graphs, the second one deals with an application to the tree alternative conjecture.

\subsection{Generalisation to graphs}\label{sec_Graphs}

Our investigations in this paper were focused on a special class of graphs: on trees.
Obviously, the following problem arises.

\begin{prob}\label{prob_graph}
Generalise our main theorems to \se s of graphs.
\end{prob}

When one looks at this problem, one naturally brings up the question about the behaviour of the ends: the \se s are injective maps on the vertex sets, but shall that extends to the ends as well?
A priori, this does not seem clear.
Let us give a short example to show that ends of graphs may collapse if we do not make further restriction on the \se s.

\begin{exam}\label{ex_graph}
Let $G$ be a complete graph with countably infinitely many vertices.
Let $x\in V(G)$.
We attach a new ray at~$x$ to obtain a new graph~$H$.
Then $H$ has two ends.
However, there is a \se\ $g$ of~$H$ that maps $G$ into a proper subgraph of~$G$ leaving infinitely many vertices of infinite degree outside of $g(G)$ and mapping the attached ray into $G\sm g(G)$.
Both ends of~$H$ are mapped by~$g$ to the same end of~$H$, the one originating from~$G$.\footnote{If we also require the \se s to preserve non-adjacency we can modify the example by subdividing every edge of the complete graph once. Then the \se\ of Example~\ref{ex_graph} induces a \se\ of the new graph that preserves not only adjacency but also non-adjacency.}
\end{exam}

Because of this possible collapse of ends under \se s, we may distinguish two cases for arbitrary graphs.
We call a \se\ of a graph \emph{strong} if it extends to an injective map on the ends and \emph{weak} otherwise.
So Problem~\ref{prob_graph} asks for a generalisation to weak \se s of graphs.
However, a main step might be to consider the following subproblem.

\begin{prob}
Generalise our main theorems to strong \se s of graphs.
\end{prob}

\subsection{Tree alternative conjecture}\label{sec_TAC}

Let $G$ and $H$ be non-isomorphic graphs.
We call $H$ a \emph{twin} of~$G$ if there are embeddings $G\to H$ and $H\to G$, i.\,e.\ injective maps $V(G)\to V(H)$ and $V(H)\to V(G)$ that preserve the adjacency relation.
Bonato and Tardif~\cite{BonTar06} made the following conjecture.

\begin{conjTA}
A tree has either none or infinitely many isomorphism classes of twins.
\end{conjTA}

As, for twins $G,H$, the embeddings $G\to H$ and $H\to G$ can be composed to a self-embedding $G\to G$, there is a natural connection to the \se s of~$G$ and Bonato and Tardif~\cite{BonTar06} suggested that the structure of the monoid of \se s may help solving their conjecture.
Indeed, Laflamme et al.~\cite{LPS-Twins} used this monoid to verify the conjecture for large classes of trees.
In order to state their result, we need some definitions.

Let $R=x_0x_1\ldots$ be a ray in a tree~$T$.
Let $T_i$ be the maximal subtree of~$T$ that is rooted at~$x_i$ and edge-disjoint from~$R$.
Let $\TF:=\{T_i\mid i\in I\}$ be a maximal set of these trees such that for no pair $T_i,T_j$ with $i\neq j$ we have embeddings $(T_i,x_i)\to (T_j,x_j)$ and $(T_j,x_j)\to (T_i,x_i)$.
If $\TF$ is finite we call $R$ \emph{regular}.
An end is \emph{regular} if it contains a regular ray.
It is easy to see that every ray in a regular end is regular.

A tree $T$ is \emph{stable} if one of the following holds.
\begin{enumerate}[(i)]
\item There is a vertex or an edge fixed by $\Emb(T)$;
\item two ends of~$T$ are fixed by $\Emb(T)$;
\item $T$ has an end that is preserved forwards and backwards by every \se\ of~$T$;
\item $T$ has a ray $R$ with $g(R)\sub R$ for all $g\in\Emb(T)$;
\item $T$ has a non-regular end preserved forwards by every \se\ of~$T$.
\end{enumerate}

\begin{thm}\cite[Theorem 1.9]{LPS-Twins}\label{thm_LPS}
The tree alternative conjecture holds for stable trees.\qed
\end{thm}

With the help of our analysis of the monoid of \se s of trees, we are able to give a precise description of the open cases of the tree alternative conjecture.

\begin{cor}\label{cor_TAC}
The tree alternative conjecture holds for all trees $T$ whose monoid $M$ of \se s does not satisfy the following properties.
\begin{enumerate}[\rm (1)]
\item\label{itm_TAC1} There is a regular end $\omega$ of~$T$ fixed by~$M$ and all elements of~$M$ preserve $\omega$ forwards.
\item\label{itm_TAC2} There is no free submonoid of~$M$ generated freely by two non-elliptic elements.
\end{enumerate}
\end{cor}

\begin{proof}
Let $T$ be a tree and $M$ be the monoid of its \se s.
By Theorem~\ref{thm_FixedPoint}, one of the following holds.
\begin{enumerate}[(i)]
\item\label{itm_FixedPoint1a} $M$ fixes either a vertex or an edge of~$T$;
\item\label{itm_FixedPoint2a} $M$ fixes a unique element of~$\L(M)$;
\item\label{itm_FixedPoint3a} $\L(M)$ consists of precisely two elements;
\item\label{itm_FixedPoint4a} $M$ contains two non-elliptic elements that do not fix the direction of the other.
\end{enumerate}

While in cases (\ref{itm_FixedPoint1a}) and (\ref{itm_FixedPoint3a}) Theorem~\ref{thm_LPS} directly implies that the tree alternative conjecture holds, case (\ref{itm_FixedPoint4a}) together with Theorem~\ref{thm_FreeMonoid} implies~(\ref{itm_TAC2}).
So the only case that remains is if $M$ fixes a unique element $\omega$ of~$\LF(M)$.

We continue by analysing~$\DF(M)$.
If $\DF(M)$ is empty, then we only have elliptic elements in~$M$.
Let $f\in M$. As $f$ is elliptic and fixes~$\omega$, it fixes a vertex and hence the ray in~$\omega$ starting at this vertex.
So all elements of~$M$ preserve $\omega$ forwards and backwards.
By Theorem~\ref{thm_LPS}, the tree alternative conjecture holds for~$T$.

If $\omega$ is the only direction, then all elements of~$M$ preserve $\omega$ forwards.
If $\omega$ is regular, then we have (\ref{itm_TAC1}) and, if $\omega$ is not regular, then $T$ is stable and the tree alternative conjecture holds by Theorem~\ref{thm_LPS}.

If there are at least two directions distinct from~$\omega$, then Theorem~\ref{thm_FreeMonoid} implies (\ref{itm_TAC2}).

It remains to consider the case that there exists precisely one direction $g^+$ distinct from~$\omega$.
Since $|\DF(M)|=2$ and $\DF(M)$ is dense in $\LF(M)$ by Theorem~\ref{thm_DenseNumber}\,(\ref{itm_DenseNumber1}), we have $|\LF(M)|=2$.
So we are in case (\ref{itm_FixedPoint3a}) and the tree alternative conjecture holds as we have already seen.
\end{proof}

\section*{Acknowledgement}

I thank M.~Pouzet for discussions on this topic.

\providecommand{\bysame}{\leavevmode\hbox to3em{\hrulefill}\thinspace}
\providecommand{\MR}{\relax\ifhmode\unskip\space\fi MR }
\providecommand{\MRhref}[2]{%
  \href{http://www.ams.org/mathscinet-getitem?mr=#1}{#2}
}
\providecommand{\href}[2]{#2}


\begin{thebibliography}{10}

\bibitem{BonTar06}
A.~Bonato and C.~Tardif, \emph{Mutually embeddable graphs and the tree
  alternative conjecture}, J.\ Combin.\ Theory Ser.\ B \textbf{96} (2006),
  874--880.

\bibitem{TopSurveyII}
R.~Diestel, \emph{Locally finite graphs with ends: a topological approach.
  {II}.\ {A}pplications}, Discrete Math. \textbf{310} (2010), 2750--2765.

\bibitem{TopSurveyI}
R.~Diestel, \emph{Locally finite graphs with ends: a topological approach.
  {I}.\ {B}asic theory}, Discrete Math. \textbf{311} (2011), 1423--1447.

\bibitem{H-AutoAndEndo}
R.~Halin, \emph{Automorphisms and {E}ndomorphisms of {I}nfinite {L}ocally
  {F}inite {G}raphs}, Abh.\ Math.\ Sem.\ Univ.\ Hamburg \textbf{39} (1973),
  no.~1, 251--283.

\bibitem{GroupsOnMetricSpaces}
M.~Hamann, \emph{Groups acting on metric spaces{:} fixed points and free
  subgroups}, Abh.\ Math.\ Sem.\ Univ.\ Hamburg (to appear), Rudolf Halin
  memorial volume, arXiv:1301.6513.

\bibitem{J-FiniteFixedSets}
H.A. Jung, \emph{On finite fixed sets in infinite graphs}, Discrete Math.
  \textbf{131} (1991), no.~1-3, 115--125.

\bibitem{LPS-Twins}
C.~Laflamme, M.~Pouzet, and N.~Sauer, \emph{Invariant subsets of scattered
  trees. an application to the tree alternative property of bonato and tardif},
  Abh.\ Math.\ Sem.\ Univ.\ Hamburg (to appear).

\bibitem{L-Personal}
F.~Lehner, \emph{Personal communication}, 2016.

\bibitem{EoG}
R.G. M\"oller, \emph{Ends of graphs}, Math.\ Proc.\ Cambridge Phil.\ Soc.
  \textbf{111} (1992), no.~2, 255--266.

\bibitem{P-BilatDenseness}
M.~Pavone, \emph{Bilateral denseness of the hyperbolic limit points of groups
  acting on metric spaces}, Abh.\ Math.\ Sem.\ Univ.\ Hamburg \textbf{67}
  (1997), no.~1, 123--135.

\bibitem{P-Personal}
M.~Pouzet, \emph{Personal communication}, 2016.

\bibitem{Woess-Amenable}
W.~Woess, \emph{Amenable group actions on infinite graphs}, Math.\ Ann.
  \textbf{284} (1989), no.~2, 251--265.

\bibitem{W-FixedSets}
\bysame, \emph{Fixed sets and free subgroups of groups acting on metric
  spaces}, Math.\ Z. \textbf{214} (1993), no.~3, 425--440.

\end{thebibliography}
\end{document}